\documentclass[11pt]{article}
\usepackage{amsfonts,amssymb,amsmath,latexsym}

\newenvironment{proof}
      {\medskip\noindent{\bf Proof.}\hspace{1mm}}
      {\hfill$\Box$\medskip}

\def\qed{\ifvmode\mbox{ }\else\unskip\fi\hskip 1em plus 10fill$\Box$}

\newtheorem{theorem}{Theorem}

\newtheorem{lemma}[theorem]{Lemma}

\newtheorem{corollary}[theorem]{Corollary}

\begin{document}

\title{Independent arithmetic progressions}
\author{David Conlon\thanks{Mathematical Institute, Oxford OX2 6GG, United Kingdom. Email: {\tt david.conlon@maths.ox.ac.uk}. Research supported by ERC Starting Grant 676632.}
\and 
Jacob Fox\thanks{Department of Mathematics, Stanford University, Stanford, CA 94305. Email: {\tt jacobfox@stanford.edu}. Research supported by a Packard Fellowship and by NSF Career Award DMS-1352121.}
\and 
Benny Sudakov\thanks{Department of Mathematics, ETH, 8092 Z\"urich, Switzerland. Email: {\tt benjamin.sudakov@math.ethz.ch}.
Research supported in part by SNSF grant 200021-175573.}}
\date{}

\maketitle

\begin{abstract}
We show that there is a positive constant $c$ such that any graph on vertex set $[n]$ with at most $c n^2/k^2 \log k$ edges contains an independent set of order $k$ whose vertices form an arithmetic progression. We also present applications of this result to several questions in Ramsey theory. 
\end{abstract}

A classical theorem of Tur\'an \cite{Tu} shows that any graph on $n$ vertices with less than $\frac{n(n-k+1)}{2(k-1)}$ edges contains an independent set of order $k$. The celebrated Szemer\'edi's theorem \cite{Sz} states that for $\delta>0$, $k \in \mathbb{N}$, and $n$ sufficiently large in terms of $k$ and $\delta$, any subset of $[n]=\{1,\ldots,n\}$ of order at least $\delta n$ contains a $k$-term arithmetic progression. Here we marry the themes of these results and deduce as consequences bounds on three other well-studied problems on rainbow arithmetic progressions and set mappings. 

Given a graph with vertex set $[n]$, a $k$-term arithmetic progression is said to be {\it independent} if it is an independent set in the graph. Our main result is a Tur\'an-type theorem, showing that any sparse graph on vertex set $[n]$ contains an independent arithmetic progression. Before proving this result, we need a standard estimate from number theory. Note that all logs will be taken to base $e$.

\begin{lemma} \label{lem:sieve}
There is a positive constant $\eta$ such that, for all $n \geq \eta^{-1} k \log k$, the number of integers from $[n]$ which are relatively prime to $1, 2, \dots, k$ is at least $\eta n/\log k$.
\end{lemma}

\begin{proof}
Writing $\Phi(x, y)$ for the number of integers less than or equal to $x$ all of whose prime factors are greater than $y$, a result of Buchstab (see Section 7.2 of~\cite{MV06}) says that
\[\Phi(x,y) = \frac{w(u) x}{\log y} - \frac{y}{\log y} + O\left(\frac{x}{\log^2 x}\right),\]
where $u$ is defined by $y = x^{1/u}$ and $w(u)$ is the Buchstab function, equal to $1/u$ for $1 < u \leq 2$ and asymptotic to $e^{-\gamma}$, with $\gamma$ the Euler--Mascheroni constant, as $u$ tends to infinity. For $k$ sufficiently large, say $k \geq k_0$, and $n \geq k \log k$, the required estimate with $\eta = 1/10$ easily follows by applying this result with $x = n$ and $y = k$. For $k < k_0$, the estimate follows by choosing $\eta$ such that $\eta^{-1} \geq \max(20 \log k_0, k_0 \log k_0)$. Then $n \geq k_0 \log k_0$, so that $\Phi(n, k) \geq \Phi(n, k_0) \geq n/10 \log k_0 \geq \eta n/\log k$.
\end{proof}

Our main result, which is tight up to the logarithmic factor, is now as follows.

\begin{theorem}\label{thm:APmain}
There is a positive constant $\varepsilon$ such that any graph $G$ on $[n]$ with less than $\varepsilon \frac{n^2}{k^2 \log k}$ edges contains a $k$-term independent arithmetic progression.
\end{theorem}

\begin{proof} 
We split into two cases, depending on the size of $n$. For $n \geq 2 \eta^{-1} k^2 \log k$, where $\eta$ is as in Lemma~\ref{lem:sieve}, we consider the set of integers $X$ which are relatively prime to $1,2,\ldots,k$ and let $\mathcal{A}$ be the set of $k$-term arithmetic progressions in $[n]$ whose difference is in $X$. We can form an arithmetic progression in $\mathcal{A}$ by choosing the first term from $[n/2]$ and the common difference from $X \cap [n/2k]$. Therefore, since $n/2k \geq \eta^{-1} k \log k$, Lemma~\ref{lem:sieve} applies to show that $|\mathcal{A}| \geq \eta n^2/4k \log k$. Each pair of integers are in arithmetic progressions with at most one common difference in $X$ and, hence, are in at most $k-1$ arithmetic progressions in $\mathcal{A}$. Thus, the number of arithmetic progressions in $\mathcal{A}$ which contain an edge of $G$ is at most $e(G)k$. Taking $\varepsilon < \eta/8$, we have that $e(G)k <  \varepsilon \frac{n^2}{k \log k} < |\mathcal{A}|$, so there is an arithmetic progression in $\mathcal{A}$ which forms an independent set. 

For the second case, when $n < 2 \eta^{-1} k^2 \log k$, we let $\mathcal{B}$ be the set of $k$-term arithmetic progressions in $[n]$ whose difference is a prime. By the same argument as in the previous case, the number of arithmetic progressions in $\mathcal{B}$ which contain an edge of $G$ is at most $e(G)k < \varepsilon \frac{n^2}{k \log k}$. On the other hand, the number of progressions in $\mathcal{B}$ is at least $\pi(n/2k) n/2$, where $\pi(x)$ is the prime counting function. Since there exist positive constants $a$ and $C$ such that $\pi(x) > a \frac{x}{\log x}$ and $2 \eta^{-1} k^2 \log k < k^C$, we have that $\pi(n/2k) n/2 > \frac{a}{2C} \frac{n^2}{k \log k}$. Therefore, for $\varepsilon < a/2C$, there is an independent arithmetic progression.
\end{proof}

In a coloring of $[n]$, an arithmetic progression is {\it rainbow} if its elements are all different colors. The {\it sub-Ramsey number $sr(m,k)$} is the minimum $n$ such that every coloring of $[n]$ in which no color is used more than $m$ times has a 
rainbow $k$-term arithmetic progression. Alon, Caro, and Tuza \cite{ACT} proved that there are constants $c,c'>0$ such that $$c'\frac{mk^2}{\log mk} \leq sr(m,k) \leq c mk^2 \log (mk).$$ They also showed that there is an upper bound on $sr(m,k)$ 
which is linear in $m$ but with a worse dependence on $k$, namely, $sr(m,k) \leq cmk^3$. The lower bound was later improved by Fox, Jungi\'c, and Radoi\v ci\' c \cite{FJR} to $sr(m,k) \geq c'mk^2$. Here we improve on the upper bounds of Alon, Caro, and Tuza \cite{ACT}. 

\begin{corollary}\label{subRamsey}
There is a constant $c$ such that the sub-Ramsey number satisfies 
$$sr(m,k) \leq cmk^2 \log k.$$
\end{corollary}

\begin{proof}
Consider a coloring of $[n]$ with $n=  \varepsilon^{-1} mk^2 \log k$, with $\varepsilon$ as in Theorem~\ref{thm:APmain}, where no color appears more than $m$ times. Define a graph on $[n]$ where two integers are adjacent if they receive the same color. The graph consists of a disjoint union of cliques of order at most $m$. Since the maximum of $\sum_i \binom{x_i}{2}$ under the constraint $\sum_i x_i$ occurs when each term is as large as possible, the number of edges in this graph is at most $\frac{n}{m}{m \choose 2}<\frac{nm}{2}$. Therefore, by our choice of $n$, the number of edges is such that Theorem \ref{thm:APmain} applies to give an independent $k$-term arithmetic progression, which is a rainbow arithmetic progression in our coloring of $[n]$. 
\end{proof}

Let $T_k$ denote the smallest positive integer $t$ such that for every positive integer $m$, every equinumerous $t$-coloring of $[tm]$ contains a rainbow $k$-term arithmetic progression. Jungi\'c et al.~\cite{JLMNR} proved that there are positive constants $c,c'$ such that $$c' k^2 \leq T_k \leq c k^3.$$ 
They conjectured that the lower bound is correct, that is, $T_k=\Theta(k^2)$, a problem which was reiterated in the survey~\cite{JNR}. 
Here we make progress on this conjecture, improving the upper bound to $c k^2 \log k$. Note that an equinumerous 
$t$-coloring of $[tm]$ uses each color exactly $m$  times, so $T_k$ is at most the maximum of $sr(m,k)/m$ over all positive integers $m$. Hence, by Corollary \ref{subRamsey}, we obtain the following corollary. 

\begin{corollary}
There is a constant $c$ such that
$$T_k \leq c k^2 \log k.$$
\end{corollary}

Motivated by the set mapping problem of Erd\H{o}s and Hajnal, Caro \cite{C87} proved that for every positive integer $k$, there is a minimum integer $n_0 = n_0(k)$ such that, for all $n \geq n_0$ and every permutation $\pi:[n] \to [n]$, there is a $k$-term arithmetic progression $A$ such that $\pi(i) \not \in A$ for all $i \in A$.  Moreover, he showed that there are constants $c, c' > 0$ such that $c' k^2/\log k \leq n_0(k) \leq k^2 2^{c \log k / \log \log k}$. Alon et al.~\cite{ACT} used the same methods they had used to bound $sr(m,k)$ to improve the earlier upper bound to $n_0(k) \leq c k^2 \log k$. Our result gives a simple alternative proof of this.

\begin{corollary}
There is a constant $c$ such that
$$n(k) \leq c k^2 \log k.$$
\end{corollary}

\begin{proof}
Consider the graph on $[n]$ with edges $(i,\pi(i))$ for $i \in [n]$. This graph has at most $n$ edges. By choosing $c$ large enough, we can make the number of edges such that Theorem \ref{thm:APmain} applies to give an independent arithmetic progression in this graph. This arithmetic progression has the required property.
\end{proof}

\vspace{2mm}
\noindent
{\bf Acknowledgements.} This note was first written in May 2015, predating a recent paper of Geneson~\cite{G18} showing that $T_k \leq k^{5/2 + o(1)}$, and will form part of the forthcoming paper {\it Short proofs of some extremal results III}. We would like to thank Kevin Ford for some helpful discussions. We would also like to mention that recently J\'ozsef Balogh, Will Linz and Mina Nahvi independently investigated the question of estimating $T_k$ and showed that $T_k = k^{2 + o(1)}$.


\begin{thebibliography}{}


\bibitem{ACT} N. Alon, Y. Caro, and Z. Tuza, Sub-Ramsey numbers for arithmetic progressions, 
{\it Graphs Combin.} {\bf 5} (1989), 307--314. 

\bibitem{C87} Y. Caro, Extremal problems concerning transformations of the edges of the complete hypergraphs, {\it J. Graph Theory} {\bf 11} (1987), 25--37.


\bibitem{FJR} J. Fox, V. Jungi\'c, and R. Radoi\v ci\' c, Sub-Ramsey numbers for arithmetic progressions and the Sidon equation, {\it Integers} {\bf 7} (2007), A12.

\bibitem{G18} J. Geneson, A note on long rainbow arithmetic progressions, arXiv:1811.07989 [math.CO].

\bibitem{JLMNR} V. Jungi\'c, J. Licht (Fox), M. Mahdian, J. Ne\v set\v ril, and R. Radoi\v ci\' c, Rainbow Arithmetic Progressions and Anti-Ramsey Results, {\it  Combin. Prob. Comput.} {\bf 12} (2003), 599--620.

\bibitem{JNR} V. Jungi\'c, J. Ne\v set\v ril, and R. Radoi\v ci\' c, Rainbow Ramsey theory, {\it Integers} {\bf 5} (2005), A9, 13 pp. 

\bibitem{MV06} H. L. Montgomery and R. C. Vaughan, {\bf Multiplicative number theory. I. Classical theory}, Cambridge Studies in Advanced Mathematics, 97, Cambridge University Press, Cambridge, 2007.

\bibitem{Sz} E. Szemer\'edi, On sets of integers containing no $k$ elements in arithmetic progression, {\it Acta Arith.} {\bf 27} (1975), 199--245.

\bibitem{Tu} P. Tur\'an, Eine Extremalaufgabe aus der Graphentheorie, {\it Mat. Fiz. Lapok} {\bf 48} (1941), 436--452.

\end{thebibliography}
\end{document}